\documentclass[]{article}
\usepackage[nomarkers, nolists]{endfloat}

\usepackage{standalone}
\usepackage{amsmath}
\usepackage{amsfonts}
\usepackage{amsthm}
\usepackage[linktoc=all]{hyperref}
\usepackage{multirow}

\theoremstyle{plain}
\newtheorem{prop}{Proposition}
\theoremstyle{definition}
\newtheorem{defn}{Definition}
\newtheorem{exmp}{Example}
\theoremstyle{remark}
\newtheorem{rem}{Remark}

\title{Solving square polynomial systems:\\ a practical method using Bezout matrices}
\author{Jean-Paul Cardinal\\
\texttt{cardinal@math.univ-paris13.fr}}

\begin{document}
\maketitle
\begin{abstract}
Let $f$ be a polynomial system consisting of $n$ polynomials $f_1,\cdots, f_n$ in $n$ variables $x_1,\cdots, x_n$, with coefficients in $\mathbb{Q}$ and let $\langle f\rangle$ be the ideal generated by $f$. Such a polynomial system, which has as many equations as variables is called a square system. It may be zero-dimensional, i.e the system of equations $f = 0$ has finitely many complex solutions, or equivalently the dimension of the quotient algebra $A = \mathbb{Q}[x]/\langle f\rangle$ is finite. In this case, the companion matrices of $f$ are defined as the matrices of the endomorphisms of $A$, called multiplication maps, $x_j : \left\vert
\begin{array}{c}
h \mapsto x_jh
\end{array}
\right.$, written in some basis of $A$.
We present a practical and efficient method to compute the companion matrices of $f$ in the case when the system is zero-dimensional. When it is not zero-dimensional, then the method works as well and still produces matrices having properties similar to the zero-dimensional case.
The whole method consists in matrix calculations. An experiment illustrates the method's effectiveness. 

\end{abstract}
\tableofcontents

\section{Univariate case}
\label{univariate}
We recall some well-known facts about univariate polynomials;
in this section we consider a polynomial $f = a_0x^d + \dots + a_{d-1}x + a_d \in \mathbb{Q}[x]$ in the variable $x$, with rational coefficients; we denote by $A = \mathbb{Q}[x]/\langle f \rangle$ the quotient algebra of $\mathbb{Q}[x]$ by the ideal $\langle f \rangle$, and we denote indifferently by $x$ the variable $x$, its projection on the quotient algebra $A$ and the multiplication map $x : \left\vert h \mapsto xh \right.$ defined on $A$. 
The special basis $\bold{x} = (1, x,\cdots, x^{d-1})$ of the vector space $A$ is called the {\bf monomial basis}.

\subsection{Mutiplication maps}
The multiplication map $x : \left\vert h \mapsto xh \right.$ is an endomorphism of $A$ which, when written in the monomial basis, has a matrix $X$ called the {\bf companion matrix} of $f$. 
The matrix $X$ is Hessenberg and writes
\begin{equation}
\label{compan}
X =
\begin{bmatrix}
	0 & \cdots & 0 & -a_d/a_0 \\
	1 & 0 & \cdots & -a_{d-1}/a_0 \\
	\vdots  & \ddots  & \ddots & \vdots  \\
	0 & \cdots & 1 & -a_1/a_0
\end{bmatrix}
\end{equation}
\begin{prop}
\label{compan2roots}
The characteristic polynomial of $X$ is $f$.
\end{prop}

\begin{rem}
We deduce from Proposition \ref{compan2roots} that the eigenvalues of $X$ are the roots of $f$, taking the multiplicities into account. 
Moreover, since the matrix $X$ is Hessenberg, we can use reliable techniques, like the QR algorithm, to compute its eigenvalues. 
This gives a practical and fast method to compute numerical approximations of the roots of $f$. 
\end{rem}

\begin{rem}
\label{g_1=g_2}
If $g_1, g_2$ are two polynomials of $\mathbb{Q}[x]$ that are equal modulo $f$, then $g_1(X) = g_2(X)$;
therefore, if $g \in A$, then the matrix $g(X)$ is defined without any ambiguity;
this matrix does not depend on the choice of the particular representative of $g$, and it is the matrix of the map $g:\left\vert h \mapsto gh \right.$, written in the monomial basis.
\end{rem}

\subsection{Bezout polynomials and Bezout matrices}
As we have seen, the companion matrix can be used to calculate the roots of a univariate polynomial. 
Interestingly, this can be naturally extended to zero-dimensional multivariate systems; 
if we have $n$ polynomials $f = f_1, \ldots, f_n$ in $n$ variables $x = x_1, \ldots, x_n$, then we simply define the companion matrices as the matrices of the multiplication maps $x_j : \left\vert h \mapsto x_jh \right.$ defined on the quotient algebra $ A = \mathbb{Q}[x]/ \langle f\rangle$, written in some basis of $A$. 
However, in the multivariate case, $A$ has no canonical basis, and the companion matrices do not have an obvious form. 
We can nonetheless resolve this problem by using a family of matrices, the so-called Bezout matrices, that exists both in the univariate case and in the multivariate case, and that serve as intermediate matrices to construct the companion matrices. 
Let's introduce the Bezout matrices in the case of univariate polynomials.

\begin{defn}
\label{def:bez}
Let $f \in \mathbb{Q}[x]$ be a fixed polynomial and let $g$ be any other polynomial. We introduce a new variable $y$ and define the {\bf Bezout polynomial} $\delta(g)$, or {\bf bezoutian}, as the polynomial in two variables $x, y$
$$
\delta(g) = \dfrac{f(x)g(y)-f(y)g(x)}{x-y}
$$
This polynomial is of degree $m-1$ in both variables $x, y$, where $m$ is the maximum of the degrees of $f$ and $g$.
If we write the bezoutian
\begin{equation}
\delta(g) = \sum_{\alpha,\beta = 0, \cdots, m-1} b_{\alpha\beta} x^\alpha y^\beta
\end{equation}
then the matrix of coefficients $B(g) = [b_{\alpha\beta}]$ is called the {\bf Bezout matrix}.
\end{defn}

\begin{rem}
The size of a Bezout matrix may be loosely defined; when working with several Bezout matrices, it may be desirable to pad some of them with extra columns or lines of zeros to get compatible sizes.
\end{rem}
\begin{rem}
The Bezout poynomial $\delta(g)$ and the Bezout matrix $B(g)$ satisfy the following equality
\begin{equation}
	\label{xBg}
	\delta(g) = \bold{x} B(g) \bold{y}^T
\end{equation}
where $\bold{x} = (1, x,\cdots, x^{m-1}) \in \mathbb{Q}[x]^m$ and $\bold{y} = (1, y,\cdots, y^{m-1}) \in \mathbb{Q}[y]^m$ are two vectors of monomials.
\end{rem}

\begin{exmp}
\label{exmp_1}
We choose $f = x^2 - 3x + 2$ as the fixed polynomial, and we examine the two cases $g=1$ and $g = x^3$. 
The Bezout polynomials are $\delta(1) = -3 + x + y$ and $\delta(x^3) = -2x^2 - 2xy -2y^2 + 3x^2y + 3xy^2 -x^2y^2$. 
The Bezout matrices $B(1)$ et $B(x^3)$ appear when we write  $\delta(1)$ and  $\delta(x^3)$ as double-entry arrays indexed by the monomials $1, x, x^2$ and $1, y, y^2$.
$$
\begin{array}{c|ccc}
\delta(1) & 1 & y & y^2\\
\hline
1 & -3 & 1 & 0\\
x & 1 & 0 & 0\\
x^2 & 0 & 0 & 0
\end{array}
\hspace{1cm}
\begin{array}{c|ccc}
\delta(x^3) & 1 & y & y^2\\
\hline
1 & 0 & 0 & -2\\
x & 0 & -2 & 3\\
x^2 & -2 & 3 & -1
\end{array}
$$
\end{exmp}

\begin{prop}
\label{relations_prop}
Let $f$ be a fixed polynomial and $g$ be another polynomial; if we denote by $m$ the maximum of the degrees of $f$ and $g$ and put $\bold{x} = (1, x,\cdots, x^{m-1})$, then
\begin{equation}
\label{relations}
	\bold{x}B(1)g = \bold{x}B(g)
\end{equation}
where the equality must be understood componentwise in $\mathbb{Q}[x]^m$ and modulo $f$.
\end{prop}
\begin{proof}
We rewrite $\delta(g)$ as
\begin{align*}
	\delta(g) & = & g(x)\dfrac{f(x)-f(y)}{x-y} - f(x)\dfrac{g(x)-g(y)}{x-y} \\ \nonumber
	\delta(g) & = & g(x)\delta(1) - f(x)\dfrac{g(x)-g(y)}{x-y}
\end{align*}
This equality holds between elements of $\mathbb{Q}[x][y]$. 
If $h\in \mathbb{Q}[x][y]$ and $\beta\in\mathbb{N}$, we denote by $h_\beta \in \mathbb{Q}[x]$ the coefficient of $y^\beta$ in the polynomial $h$; then
$$\delta(g)_\beta = g(x)\delta(1)_\beta - f(x)\left(\dfrac{g(x)-g(y)}{x-y}\right)_\beta$$
which holds in $\mathbb{Q}[x]$. Thus, we have $\delta(g)_\beta = g(x)\delta(1)_\beta$ modulo $f$; since this is true for all $\beta\in\mathbb{N}$, Equality (\ref{relations}) follows.
\end{proof}

\begin{rem}
Each column of a Bezout matrix, when left-multiplied by $\bold{x}$, is a polynomial in the variable $x$; when this does not bring to confusion, we will think of columns of a Bezout matrix as elements of $\mathbb{Q}[x]$, and lines as elements of $\mathbb{Q}[y]$.
Saying Proposition \ref{relations_prop} differently : each column of $B(1)$, when multiplied by $g$, equals the column of same index of $B(g)$, modulo $f$.
\end{rem}

\begin{exmp}
Returning to Example \ref{exmp_1}, Proposition \ref{relations_prop} says that the following equalities hold, modulo $f$
\begin{align*}
 (-3 + x)x^3 &= -2x^2, \\
 (1)x^3 &= -2x + 3x^2, \\
 (0)x^3 &= -2 + 3x - x^2
\end{align*}
\end{exmp}

\begin{rem}
If we work with lines instead of columns, then Proposition \ref{relations_prop} says that $gB(1)\bold{y}^T = B(g)\bold{y}^T$, giving equalities in $\mathbb{Q}[y]/\langle f \rangle$.
\end{rem}

\subsection{Bezout matrices are related to the companion matrix}
\label{Bar}
Given a fixed polynomial $f$ of degree $d$, the bezoutians $\delta(1)$ and $\delta(x)$ write
\begin{equation}
	\begin{array}{c|cccc}
		\delta(1) & 1 & y & \dots & y^{d-1} \\
		\hline
		1 & a_{d-1} & \ldots & \dots & a_0 \\
		x & a_{d-2} & \dots & a_0 & 0 \\
		\vdots & \vdots & \vdots & \vdots & \vdots \\
		x_{d-1} & a_0 & 0 & \ldots & 0 \\
	\end{array}
	\hspace{1.5cm}
	\begin{array}{c|cccc}
		\delta(x) & 1 & y & \dots & y^{d-1} \\
		\hline
		1 & -a_{d} & 0 & \dots & 0 \\
		x & 0 & a_{d-2} & \ldots & a_0 \\
		\vdots & \vdots & \vdots & \vdots & \vdots \\
		x_{d-1} & 0 & a_0 & \ldots & 0 \\
	\end{array}.
\end{equation}
The corresponding arrays of coefficients are the Bezout matrices $B(1)$, which is clearly invertible, and $B(x)$. 
These two matrices are specially important because they are related to the companion matrix:
\begin{prop}
\label{Barnett}
The compagnon matrix $X$ is related to the Bezout matrices $B(x), B(1)$ by the {\bf Barnett decomposition formula}
\cite{Barnett}
\begin{equation}
	X = B(x)B(1)^{-1}
\end{equation}
\end{prop}

\begin{proof}
Let us consider the two vectors of elements of $A$
\begin{equation}
	\begin{array}{lll}
		\bold{x}B(1) & = & (a_{d-1} + a_{d-2}x + \cdots + a_0x^{d-1}, \cdots, a_1 + a_0x,  a_0).\\
		\bold{x}B(x) & = & (-a_d, a_{d-2}x + \cdots + a_0x^{d-1}, \cdots, a_0x)
	\end{array}
\end{equation}
and put $\hat{\bold{x}} = \bold{x}B(1)$.
As $B(1)$ is invertible, $\hat{\bold{x}}$ is a basis of the vector space $A^d$, called the {\bf Horner basis}.
From Proposition \ref{relations_prop}, we have $\hat{\bold{x}}x = \bold{x}B(1)$. 
By construction, $B(1)$ is the matrix of the Horner basis $\hat{\bold{x}}$ written on the monomial basis $\bold{x}$, and $B(x)$ is the matrix of $\hat{\bold{x}}x$ written on the monomial basis;
$B(1)^{-1}B(x)$ is thus the matrix of $\hat{\bold{x}}x$ written on the Horner basis $\hat{\bold{x}}$. 
This means that the multiplication map $\left\vert h \mapsto xh\right.$ is represented in the Horner basis $\hat{\bold{x}}$ by the matrix $B(1)^{-1}B(x)$, and is therefore represented in the monomial basis $\bold{x}$ by the matrix $B(1)(B(1)^{-1}B(x))B(1)^{-1} = B(x)B(1)^{-1}$.
\end{proof}

\subsection{Barnett decomposition formula}
\label{Bar_gen}
The Barnett decomposition formula relates the companion matrix, representing the multiplication map $x$, to the Bezout matrices of the polynomials $1$ and $x$; this can be naturally extended as follows.
Let $g \in \mathbb{Q}[x]$ be any polynomial; matrix $g(X)$ is related to the Bezout matrices $B(1)$ et $B(g)$ by the so-called {\bf general Barnett decomposition formula}
\begin{equation}
	\label{GBG}
	g(X) = B(g)B(1)^{-1}
\end{equation}
given that if the sizes of $B(1)$ and $B(g)$ differ then we must transform and resize $B(g)$ according to the following procedure, explained on an example.
Formula~(\ref{GBG}) is easily checked when the degree of $g$ is smaller or equal to the degree of $f$ because $B(1)$ and $B(g)$ have the same size;
for example, if $f = x^2 - 3x + 2$, then
$$
\begin{array}{c|cc}
	\delta(1) & 1 & y \\
	\hline
	1 & -3 & 1 \\
	x & 1 & 0
\end{array}
\hspace{1cm}
\begin{array}{c|cc}
	\delta(x) & 1 & y \\
	\hline
	1 & -2 & 1 \\
	x & 1 & 0
\end{array}
\hspace{1cm}
\begin{array}{c|cc}
	\delta(x^2) & 1 & y \\
	\hline
	1 & 0 & -2 \\
	x & -2 & 3
\end{array}
$$
thus
\begin{equation}
	X = B(x)B(1)^{-1} =
	\begin{bmatrix}
		0 & -2 \\
		1 & 3
	\end{bmatrix}
	\hspace{1cm}
	X^2 = B(x^2)B(1)^{-1} =
	\begin{bmatrix}
		-3 & -6 \\
		2 & 7
	\end{bmatrix}
\end{equation}
which is consistent with formula~(\ref{GBG}).
However, if the degree $m$ of $g$ is strictly larger than the degree $d$ of $f$, then the sizes of $B(g)$ and $B(1)$ differ and the product $B(g)B(1)^{-1}$ no longer makes sense. 
This can be fixed by indexing the Bezout matrices by the same monomials, namely $\bold{x} = (1, x,\cdots, x^{m-1})$ and $\bold{y} = (1, y,\cdots, y^{m-1})$. 
For example, with $f$ as above and $g = x^3$, we have
$$
\begin{array}{c|ccc}
\delta(1) & 1 & y & y^2\\
\hline
1 & -3 & 1 & 0\\
x & 1 & 0 & 0\\
x^2 & 0 & 0 & 0
\end{array}
\hspace{1cm}
\begin{array}{c|ccc}
\delta(x^3) & 1 & y & y^2\\
\hline
1 & 0 & 0 & -2\\
x & 0 & -2 & 3\\
x^2 & -2 & 3 & -1
\end{array}
$$
In doing so, $B(1)$ is no longer invertible; the key to obtain simultaneously matrices of equal size and the invertibility of $B(1)$ is to write both bezoutians modulo $f$:
\begin{align} \nonumber 
	\delta(x^3) &=
	\begin{bmatrix}
			1 & x & x^2
	\end{bmatrix}
	\begin{bmatrix}
		0 & 0 & -2 \\
		0 & -2 & 3 \\
		-2 & 3 & -1
	\end{bmatrix}
	\begin{bmatrix}
		1 \\
		y \\
		y^2
	\end{bmatrix} \\ \nonumber 
	\delta(x^3) &=
	\begin{bmatrix}
		1 & x & x^2
	\end{bmatrix}
	\begin{bmatrix}
		1 & 0 & 2 \\
		0 & 1 & -3 \\
		0 & 0 & 1
	\end{bmatrix}
	\begin{bmatrix}
		1 & 0 & -2 \\
		0 & 1 & 3 \\
		0 & 0 & 1
	\end{bmatrix}
	\begin{bmatrix}
		0 & 0 & -2 \\
		0 & -2 & 3 \\
		-2 & 3 & -1
	\end{bmatrix}
	\begin{bmatrix}
		1 \\
		y \\
		y^2
	\end{bmatrix} \\ \nonumber 
	\delta(x^3) &=
	\begin{bmatrix}
			1 & x & 2 - 3x + x^2
	\end{bmatrix}
	\begin{bmatrix}
		4 & -6 & 0 \\
		-6 & 7 & 0 \\
		-2 & 3 & -1
	\end{bmatrix}
	\begin{bmatrix}
		1 \\
		y \\
		y^2
	\end{bmatrix} \\ \nonumber 
\end{align}
To sum up, we have post-multiplied the row vector
$\begin{bmatrix}
	1 & x & x^2
\end{bmatrix}$ by the Gauss transform
$$P =
\begin{bmatrix}
	1 & 0 & 2 \\
	0 & 1 & -3 \\
	0 & 0 & 1
\end{bmatrix}$$
and pre-multiplied $B(1)$ and $B(g)$ by $P^{-1}$; so the bezoutians write
$$
\begin{array}{c|ccc}
	\delta(1) & 1 & y & y^2\\
	\hline
	1 & -3 & 1 & 0\\
	x & 1 & 0 & 0\\
	2 - 3x + x^2 & 0 & 0 & 0
\end{array}
\hspace{1cm}
\begin{array}{c|ccc}
	\delta(x^3) & 1 & y & y^2\\
	\hline
	1 & 4 & -6 & 0 \\
	x & -6 & 7 & 0 \\
	2 - 3x + x^2 & -2 & 3 & -1
\end{array}
$$
According to the relations~(\ref{relations}) the third column of $\delta(x^3)$, which is $-2 + 3x - x^2$, is zero modulo $f$; we recognize the simple fact $-f = 0$. Thus,
\begin{align} \nonumber 
\delta(1) &= \begin{bmatrix}
	1 & x
\end{bmatrix}
\begin{bmatrix}
	-3 & 1 \\
	1 & 0
\end{bmatrix} \nonumber 
\begin{bmatrix}
	1 \\
	y
\end{bmatrix}\\
\delta(x^3) &= \begin{bmatrix}
	1 & x
\end{bmatrix}
\begin{bmatrix}
	4 & -6 \\
	-6 & 7
\end{bmatrix} \nonumber 
\begin{bmatrix}
	1 \\
	y
\end{bmatrix} + (2 - 3x + x^2)(-2 + 3y - y^2)
\end{align}
So the bezoutians, when reduced modulo $f$, write
$$
\begin{array}{c|cc}
	\delta(1) & 1 & y \\
	\hline
	1 & -3 & 1 \\
	x & 1 & 0
\end{array}
\hspace{1cm}
\begin{array}{c|cc}
	\delta(x^3) & 1 & y \\
	\hline
	1 & 4 & -6  \\
	x & -6 & 7
\end{array}
$$
We have obtained two Bezout matrices of equal size, with $B(1)$ invertible, that satisfy the general Barnett decomposition formula~(\ref{GBG}):
\begin{equation}
	B(x^3)B(1)^{-1} =
	\begin{bmatrix}
		-6 & -14 \\
		7 & 15
	\end{bmatrix}
	= X^3
\end{equation}

\begin{rem}
Instead of a Gauss matrix, we may use any matrix that maps a given column vector to a column vector containing just one non-zero entry, such as, for example, a Householder orthogonal matrix; we use Householder matrices in the practical implementation given in \cite{jp_code}.
\end{rem}

\section{Multivariate case}
\label{multivariate}

For a univariate polynomial $f$, the structure of the quotient algebra $A$ consists of the monomial basis and the companion matrix; this matrix is obtained either by a direct reading of the coefficients of $f$, or by taking the ratio of the two matrices $B(1), B(x)$; in contrast, for a multivariate polynomial system, neither a basis nor the companion matrices (matrices of the multiplication maps
$x_j : \left\vert
\begin{array}{c}
h \mapsto x_jh
\end{array}
\right.$ in the basis) can be read from the coefficients of the given polynomials. It is easy, however, to construct the Bezout matrices $B(1), B(x_1), \cdots, B(x_n)$, from which one can derive a basis of $A$ and the related companion matrices $X_1,\cdots, X_n$. 

Given $n$ polynomials $f_1,\cdots, f_n$ in $n$ variables $x_1,\cdots, x_n$, with coefficients in $\mathbb{Q}$, we denote by
\begin{itemize}
\item $\mathbb{Q}[x]$ the ring of polynomials in the variables $x = x_1,\cdots, x_n$,
\item $\langle f \rangle$ the ideal generated by $f = f_1,\cdots, f_n$,
\item $A = \mathbb{Q}[x]/\langle f\rangle$ the quotient algebra.
\end{itemize}
From now on we assume that $\langle f\rangle$ is {\bf zero-dimensional}, that is to say the vector space $A$ is finite dimensional \cite[p.~234]{clo}. This is always the case when $n = 1$.

\subsection{Construction of Bezout polynomials and Bezout matrices}

\subsubsection{Definition of Bezout polynomials and Bezout matrices in the multivariate case}

\begin{defn}
Let $x^\gamma = x_1^{\gamma_1}\cdots x_n^{\gamma_n} \in \mathbb{Q}[x]$ be some monomial.
We introduce a new variable set $y = y_1,\cdots, y_n$ and consider, for each couple of indices $i, j = 1\cdots n$, the ratio
\begin{equation}
\label{finite_diff}
\delta_{i,j}(x^\gamma) = \dfrac{y_j^{\gamma_j}f_i(y_1,\cdots, y_{j-1},x_j,\cdots,x_n) - x_j^{\gamma_j}f_i(y_1,\cdots,y_j,x_{j+1},\cdots,x_n)}{x_j - y_j}
\end{equation}
which is a polynomial in the variables $x, y$. 
With the $\delta_{i,j}$'s we form a finite difference matrix
\begin{equation}
\label{Delta}
\Delta(x^\gamma) = (\delta_{ij}(x^\gamma))_{ij}
\end{equation}
whose determinant $\delta(x^\gamma) = det(\Delta(x^\gamma))$, a polynomial in both variables $x, y$, is called the {\bf Bezout polynomial}, or {\bf bezoutian}, of $x^\gamma$.
This definition can be extended by linearity to a more general polynomial $g = \sum_\gamma g_\gamma x^\gamma \in \mathbb{Q}[x]$ by the formula
$$\delta(g) = \sum_\gamma g_\gamma \delta(x^\gamma)$$
Writting $\delta(g)$ as a sum of monomials, 
\begin{equation}
\label{def_bez}
\delta(g) = \sum_{0 \le \alpha,\beta} b_{\alpha\beta} x^\alpha y^\beta
\end{equation} 
we define $B(g)$, the {\bf Bezout matrix} of $g$, as the matrix of the coefficients $B(g) = [b_{\alpha\beta}]$. If we denote by $\bold{x}$ and $\bold{y}$ the sets of all the monomials $x^\alpha$ et $y^\beta$ that appear in~(\ref{def_bez}), then we have the relation, similar to~(\ref{xBg})
\begin{equation}
	\delta(g) = \bold{x} B(g) \bold{y}^T
\end{equation}
\end{defn}
The following example (see \cite{jpc} p.58) illustrates the previous definition

\begin{exmp}
\label{bez_multi}
We take $n = 2$, $f_1 = x_1^2 + x_1x_2^2 - 1, f_2 = x_1^2x_2 + x_1$ and want to calculate the Bezout matrices $B(1), B(x_1), B(x_2)$, which are useful for computing the companion matrices $X_1, X_2$. To begin with, we compute the finite difference matrices, as defined in (\ref{Delta})
\begin{align}
\Delta(1) &=
\begin{pmatrix}
x_1 + x_2^2 + y_1 & x_2y_1 + y_1y_2 \\
1 + x_1x_2 + x_2y_1 & y_1^2
\end{pmatrix} \nonumber  \\
\Delta(x_1) &=
\begin{pmatrix}
1 + x_1y_1 & x_2y_1 + y_1y_2 \\
1 + x_1x_2 + x_2y_1 & y_1^2
\end{pmatrix} \nonumber  \\
\Delta(x_2) &=
\begin{pmatrix}
x_1 + x_2^2 + y_1 & 1 - y_1^2 + x_2y_1y_2 \\
1 + x_1x_2 + x_2y_1  & -y_1
\end{pmatrix} \nonumber
\end{align}
whose determinants are the bezoutians
\begin{align}
\delta(1) &= -x_2y_1 - x_1x_2^2y_1 + x_1y_1^2 + y_1^3 - y_1y_2 - x_1x_2y_1y_2 - x_2y_1^2y_2 \nonumber \\
\delta(x_1) &=  y_1^2 - x_1x_2^2y_1^2 + x_1y_1^3 - x_1x_2y_1^2y_2 \nonumber \\
\delta(x_2) &= -1 - x_1x_2 - x_1y_1 -x_2y_1 - x_2^2y_1 + x_1x_2y_1^2 + x_2y_1^3 - x_2y_1y_2 - x_1x_2^2y_1y_2 - x_2^2y_1^2y_2\nonumber
\end{align}
The monomial appearing in these polynomials are
$\bold{x} = (1, x_2, x_2^2, x_1, x_1x_2, x_1x_2^2)$ and $\bold{y} = (1, y_1, y_1y_2, y_1^2, y_1^2y_2, y_1^3)$; the Bezout matrices $B(1), B(x_1), B(x_2)$ appear when we write these bezoutians as double-entry arrays indexed by $\bold{x}, \bold{y}$
$$\begin{array}{c|cccccc}
	\delta(1) & 1 & y_1 & y_1y_2 & y_1^2 & y_1^2y_2 & y_1^3 \\
	\hline
	1 &  &  & -1 &  &  & 1\\
	x_2 &  & -1 &  &  & -1 & \\
	x_2^2 &  &  &  &  &  & \\
	x_1 &  &  &  & 1 &  & \\
	x_1x_2 &  &  & -1 &  &  & \\
	x_1x_2^2 &  & -1 &  &  &  &
\end{array}$$
$$\begin{array}{c|cccccc}
	\delta(x_1) & 1 & y_1 & y_1y_2 & y_1^2 & y_1^2y_2 & y_1^3 \\
	\hline
	1 &  &  &  & 1 &  & \\
	x_2 &  &  &  &  &  & \\
	x_2^2 &  &  &  &  &  & \\
	x_1 &  &  &  &  &  & 1\\
	x_1x_2 &  &  &  &  & -1 & \\
	x_1x_2^2 &  &  &  & -1 &  &
\end{array}$$
$$\begin{array}{c|cccccc}
	\delta(x_2) & 1 & y_1 & y_1y_2 & y_1^2 & y_1^2y_2 & y_1^3 \\
	\hline
	1 & -1 &  &  &  &  & \\
	x_2 &  & -1 & -1 &  &  & 1\\
	x_2^2 &  & -1 &  &  & -1 & \\
	x_1 &  & -1 &  &  &  & \\
	x_1x_2 & -1 &  &  & 1 &  & \\
	x_1x_2^2 &  &  & -1 &  &  &
\end{array}$$

\end{exmp}

\begin{rem}
Contrasting with the univariate case, $\bold{x}$ and $\bold{y}$ are not bases of the vector space $A$. 
But, as we will see, they are  generating sets and we will examine how to build bases from them.
\end{rem}

\subsubsection{Practical computation of the Bezout matrices}
From now on, we will denote $\delta^{(k)}$ the Bezout polynomial $\delta(x_k)$ and $B^{(k)}$ the Bezout matrix $B(x_k)$; we will also adopt the convention that $x_0 = 1$. Accordingly, $\Delta(1), \Delta(x_1), \Delta(x_2)$ will be denoted by $\Delta^{(0)}, \Delta^{(1)}, \Delta^{(2)}$.

In the previous example, the matrices $\Delta^{(0)}, \Delta^{(1)}, \Delta^{(2)}$ have size $2$ and their entries are polynomials in $x_1, x_2$; calculating their determinant is easy. But, if either the number  $n$ of variables or the degree of the input polynomials $f_i$ increase, then this calculation becomes impractical because one cannot apply the Gauss pivot algorithm to a polynomial entries matrix. We can however overcome this difficulty by applying the following evaluation-interpolation process :
\begin{enumerate}
\item
Estimate a priori the set of monomials $x^\alpha y^\beta$ appearing in the Bezout polynomial $\delta(x_k)$ (for fixed $k$, $k = 0,\cdots, n$, ). Specifically, if we suppose that the polynomial system $f$ has multi-degree $(d_1, \cdots, d_n)$, that is to say, the partial degree in the variable $x_j$, for every $f_i, i = 1\cdots n$, is smaller or equal to $d_j$, then the polynomial $\delta(x_k)$ has multi-degree $(d_1, 2d_2, \cdots, nd_n)$ in variable $x$ and multi-degree $(nd_1, (n-1)d_2, \cdots, d_n)$ in variable $y$. Note that this estimate does not depend on $k = 0\cdots n$.
\item
Evaluate the polynomial-entries matrix $\Delta(x_k)$ on all pairs $(u, v)$, $u = (u_1,\cdots, u_n) \in U$, $v = (v_1,\cdots, v_n) \in V$, where $U, V$ are suitable sets of Fourier multi-points. We can choose, for instance, $U = \prod_{j=1..n} U_j$ where $U_j$ is the set of complex roots of the polynomial $X^{jd_j} - 1$, and $V = \prod_{j=1..n} V_j$ in such a way that $U_j$ et $V_j$ are disjoint sets, so that the denominator of the ratio (\ref{finite_diff}) never vanishes. This is obtained, for example, if $V_j$ is the set of complex roots of the polynomial $X^{(n-j+1)d_j} - \theta_j$ with $\theta_j = e^{i\pi/j}$ (here $i$ means the square root of $-1$). 
\item
For each $(u, v) \in U\times V$, apply the Gauss pivot method to the numerical matrix $\Delta(x_k)(u, v)$, of size $n\times n$, in order to compute its determinant $\delta(x_k)(u, v)$. This produces a matrix $C^{(k)}$, indexed by $(u, v) \in U\times V$, and defined by $C^{(k)}_{u, v} = \delta(x_k)(u, v)$.
\item
Interpolate the set of numerical values $\delta(x_k)(u, v), (u, v) \in U\times V$ in order to obtain the Bezout polynomial $\delta^{(k)}$ and the Bezout matrix $B^{(k)}$. The matrices $B^{(k)}$ and $C^{(k)}$ are closely related. Indeed, $B^{(k)} = \left[b^{(k)}_{\alpha\beta}\right]$ satisfies $\delta^{(k)}(x, y) = \sum_{\alpha,\beta} b^{(k)}_{\alpha\beta} x^\alpha y^\beta$, thus $C^{(k)}_{u,v} = \delta^{(k)}(u, v) = \sum_{\alpha,\beta} b^{(k)}_{\alpha\beta} u^\alpha v^\beta$. This writes as a matrix product $\left[C^{(k)}_{u,v}\right]_{u,v} = \left[u^\alpha\right]_{u,\alpha} \left[b^{(k)}_{\alpha,\beta}\right]_{\alpha, \beta} \left[v^\beta\right]_{v, \beta}^T$. If we define the Fourier matrices $F_u = \left[ u^\alpha \right]_{u, \alpha}$ and $F_v = \left[ v^\beta \right]_{v, \beta}$, then we get the evaluation-interpolation relation between matrices $B^{(k)}$ and $C^{(k)}$
\begin{equation}
C^{(k)} = F_uB^{(k)} F_v^T
 \end{equation}
Since $U$ and $V$ consist of Fourier points, $F_u$ et $F_v$ are (unitary) Fourier matrices and $B^{(k)}$ appears to be the 2D DFT of $C^{(k)}$
 \begin{equation}
 B^{(k)} = F_u^*C^{(k)} \overline{F_v}
 \end{equation}
\end{enumerate}
The computation of the Bezout matrices, as described above, have been implemented in Numpy and can be found at \cite{jp_code}.

\subsection{Barnett decomposition formula and structure of the quotient algebra.}
Since the ideal $\langle f\rangle$ is zero-dimensional, the dimension of the quotient algebra $A = \mathbb{Q}[\bold{x}]/\langle f\rangle$ is finite; we may look for some basis and its related companion matrices (matrices of the multiplication maps by $x_1,\cdots, x_n$). 
For this purpose, we will adapt the process described in Section \ref{Bar_gen} but, before, we will specify a number of algebraic properties about polynomial $\delta^0$ and matrices $B^{(k)}$.

\subsubsection{Algebraic properties of Bezout polynomials and Bezout matrices}
The following properties are simple; for a proof, the interested reader may refer to \cite{jpc}. As in Proposition \ref{Barnett}, we define families of elements of $A$ by forming the vector-matrix products
\begin{equation}
		\hat{\bold{x}}_k  =  \bold{x}B^{(k)}, \quad k=0\cdots n
\end{equation}
where $\bold{x}$ is the set of all the monomials $x^\alpha$ appearing in the Bezout polynomials $\delta^{(k)}, \quad k=0\cdots n$. We also write $\hat{\bold{x}}_0 = \hat{\bold{x}}$.
\begin{exmp}
Following Example \ref{bez_multi} we have
\begin{equation}
	\begin{array}{lll}
		\hat{\bold{x}} & = & (0, -x_2 - x_1x_2^2, -1 - x_1x_2, x_1, -x_2, 1) \\
		\hat{\bold{x}}_1 & = & (0, 0, 0, -1 - x_2^2, -x_1x_2, x_1) \\
		\hat{\bold{x}}_2 & = & (-1 - x_1x_2, -x_2 - x_2^2 - x_1, - x_2 - x_1x_2^2, x_1x_2, -x_2^2, x_2)
	\end{array}
\end{equation}
\end{exmp}

\begin{prop}
\label{xj} (see \cite{jpc}).
For all $k=1\cdots n$ we have, modulo $\langle f \rangle$,
\begin{equation}
    \hat{\bold{x}}x_k = \hat{\bold{x}}_k
\end{equation}
\end{prop}
This proposition can be easily checked on Example \ref{bez_multi}. 

So far, the univariate and multivariate cases are very similar, except on one point : in the multivariate case the families $\bold{x}$ and $\hat{\bold{x}}$ are not, in general, bases of the vector space $A$. We have, however, the weaker result (see \cite{jpc}) :

\begin{prop}
Both $\bold{x}$ and $\hat{\bold{x}}$ are generating families of the vector space $A$.
\end{prop}

\subsubsection{Reduction process}
\label{sec:reduction_process}
Following the matrix handlings described in Section \ref{Bar_gen}, we will show how to compute, starting from the generating families $\bold{x}, \hat{\bold{x}}$ and the Bezout matrices $B^{(k)}$, a basis of $A$ and the companion matrices $X_k$.
Let us illustrate this process on Example \ref{bez_multi}.\\
The first column is zero in $B^{(0)}$ but not in $B^{(2)}$; this gives, modulo $\langle f \rangle$, the relation $1 + x_1x_2 = 0$. Then, we right-multiply $\bold{x}$ by the Gauss matrix $P$ whose fifth column is $(1, 0, 0, 0, 1, 0)^{T}$ and left-multiply the Bezout matrices by $P^{-1}$; the Bezout polynomials write 
$$
\begin{array}{c|cccccc}
	\delta^{(0)} & 1 & y_1 & y_1y_2 & y_1^2 & y_1^2y_2 & y_1^3 \\
	\hline
	1 &  &  &  &  &  & 1\\
	x_2 &  & -1 &  &  & -1 & \\
	x_2^2 &  &  &  &  &  & \\
	x_1 &  &  &  & 1 &  & \\
	1+x_1x_2 &  &  & -1 &  &  & \\
	x_1x_2^2 &  & -1 &  &  &  &
\end{array}$$
$$
\begin{array}{c|cccccc}
	\delta^{(1)} & 1 & y_1 & y_1y_2 & y_1^2 & y_1^2y_2 & y_1^3 \\
	\hline
	1 &  &  &  & 1 & 1 & \\
	x_2 &  &  &  &  &  & \\
	x_2^2 &  &  &  &  &  & \\
	x_1 &  &  &  &  &  & 1\\
	1+x_1x_2 &  &  &  &  & -1 & \\
	x_1x_2^2 &  &  &  & -1 &  &
\end{array}
\hspace{0.2cm}
\begin{array}{c|cccccc}
	\delta^{(2)} & 1 & y_1 & y_1y_2 & y_1^2 & y_1^2y_2 & y_1^3 \\
	\hline
	1 &  &  &  & -1 &  & \\
	x_2 &  & -1 & -1 &  &  & 1\\
	x_2^2 &  & -1 &  &  & -1 & \\
	x_1 &  & -1 &  &  &  & \\
	1+x_1x_2 & -1 &  &  & 1 &  & \\
	x_1x_2^2 &  &  & -1 &  &  &
\end{array}
$$
Since $1 + x_1x_2 = 0$, modulo $\langle f \rangle$, we can remove the fifth row in the Bezout matrices; the Bezout polynomials write
$$
\begin{array}{c|ccccc}
	\delta^{(0)} & y_1 & y_1y_2 & y_1^2 & y_1^2y_2 & y_1^3 \\
	\hline
	1  &  &  &  &  & 1 \\
	x_2  & -1 &  &  & -1 & \\
	x_2^2  &  &  &  &  & \\
	x_1  &  &  & 1 &  & \\
	x_1x_2^2  & -1 &  &  &  &
\end{array}$$
$$
\begin{array}{c|ccccc}
	\delta^{(1)}  & y_1 & y_1y_2 & y_1^2 & y_1^2y_2 & y_1^3 \\
	\hline
	1  &  &  & 1 & 1 & \\
	x_2  &  &  &  &  & \\
	x_2^2  &  &  &  &  & \\
	x_1  &  &  &  &  & 1 \\
	x_1x_2^2  &  &  & -1 &  &
\end{array}
\hspace{0.2cm}
\begin{array}{c|ccccc}
	\delta^{(2)} & y_1 & y_1y_2 & y_1^2 & y_1^2y_2 & y_1^3 \\
	\hline
	1  &  &  & -1 &  & \\
	x_2  & -1 & -1 &  &  & 1 \\
	x_2^2  & -1 &  &  & -1 & \\
	x_1  & -1 &  &  &  & \\
	x_1x_2^2 &  & -1 &  &  &
\end{array}
$$
Now, the second column is zero in $B^{(0)}$ but not in $B^{(2)}$. We have $x_2 + x_1x_2^{2} = 0$, modulo $\langle f \rangle$ . We repeat the previous step with the Gauss matrix $P$ whose fifth column is $(0, 1, 0, 0, 1)^{T}$; the Bezout polynomials write
$$
\begin{array}{c|ccccc}
	\delta^{(0)} & y_1 & y_1y_2 & y_1^2 & y_1^2y_2 & y_1^3 \\
	\hline
	1  &  &  &  &  & 1 \\
	x_2  &  &  &  & -1 & \\
	x_2^2  &  &  &  &  & \\
	x_1  &  &  & 1 &  & \\
	x_2 + x_1x_2^2  & -1 &  &  &  &
\end{array}$$
$$
\begin{array}{c|ccccc}
	\delta^{(1)}  & y_1 & y_1y_2 & y_1^2 & y_1^2y_2 & y_1^3 \\
	\hline
	1  &  &  & 1 & 1 & \\
	x_2  &  &  & 1 &  & \\
	x_2^2  &  &  &  &  & \\
	x_1  &  &  &  &  & 1 \\
	x_2 + x_1x_2^2  &  &  & -1 &  &
\end{array}
\hspace{0.2cm}
\begin{array}{c|ccccc}
	\delta^{(2)} & y_1 & y_1y_2 & y_1^2 & y_1^2y_2 & y_1^3 \\
	\hline
	1  &  &  & -1 &  & \\
	x_2  & -1 &  &  &  & 1 \\
	x_2^2  & -1 &  &  & -1 & \\
	x_1  & -1 &  &  &  & \\
	x_2 + x_1x_2^2 &  & -1 &  &  &
\end{array}
$$
Since $x_2 + x_1x_2^{2} = 0$, modulo $\langle f \rangle$, we can remove the fifth row in each Bezout matrix; the Bezout polynomials write
$$
\begin{array}{c|cccc}
	\delta^{(0)} & y_1 & y_1^2 & y_1^2y_2 & y_1^3 \\
	\hline
	1  &   &  &  & 1 \\
	x_2  &  &  & -1 & \\
	x_2^2  &  &  &  & \\
	x_1  &  & 1 &  &
\end{array}
\hspace{0.2cm}
\begin{array}{c|cccc}
	\delta^{(1)}  & y_1 & y_1^2 & y_1^2y_2 & y_1^3 \\
	\hline
	1  &  & 1 & 1 & \\
	x_2  &  & 1 &  & \\
	x_2^2  &  &  &  & \\
	x_1  &  &  &  & 1
\end{array}
\hspace{0.2cm}
\begin{array}{c|cccc}
	\delta^{(2)} & y_1 & y_1^2 & y_1^2y_2 & y_1^3 \\
	\hline
	1  &  & -1 &  & \\
	x_2  & -1 &  &  & 1 \\
	x_2^2  & -1 &  & -1 & \\
	x_1  & -1 &  &  &
\end{array}
$$
Now, the first column is zero in $B^{(0)}$ but not in $B^{(2)}$. We have $x_2 + x_2^{2} + x_1 = 0$, modulo $\langle f \rangle$. We use the Gauss matrix $P$ whose fourth column is $(0, 1, 1, 1)^{T}$; the Bezout polynomials write
$$
\begin{array}{c|cccc}
	\delta^{(0)} & y_1 & y_1^2 & y_1^2y_2 & y_1^3 \\
	\hline
	1  &   &  &  & 1 \\
	x_2  &  & -1 & -1 & \\
	x_2^2  &  & -1 &  & \\
	x_2 + x_2^{2} + x_1  &  & 1 &  &
\end{array}$$
$$
\begin{array}{c|cccc}
	\delta^{(1)} & y_1 & y_1^2 & y_1^2y_2 & y_1^3 \\
	\hline
	1  &  & 1 & 1 & \\
	x_2  &  &  &  & \\
	x_2^2  &  &  &  & -1 \\
	x_2 + x_2^{2} + x_1  &  &  &  & 1
\end{array}
\hspace{0.2cm}
\begin{array}{c|cccc}
	\delta^{(2)} & y_1 & y_1^2 & y_1^2y_2 & y_1^3 \\
	\hline
	1  &  & -1 &  & \\
	x_2  &  &  &  & 1 \\
	x_2^2  &  &  & -1 & \\
	x_2 + x_2^{2} + x_1  & -1 &  &  &
\end{array}$$

Since $x_2 + x_2^{2} + x_1 = 0$, modulo $\langle f \rangle$, we can remove the fourth row in each Bezout matrix; the Bezout polynomials write
$$
\begin{array}{c|ccc}
	\delta^{(0)} & y_1^2 & y_1^2y_2 & y_1^3 \\
	\hline
	1  &  &  & 1 \\
	x_2  & -1 & -1 & \\
	x_2^2 & -1 &  &
\end{array}
\hspace{0.2cm}
\begin{array}{c|ccc}
	\delta^{(1)} & y_1^2 & y_1^2y_2 & y_1^3 \\
	\hline
	1  & 1 & 1 & \\
	x_2  & 1 &  & -1\\
	x_2^2  &  &  & -1
\end{array}
\hspace{0.2cm}
\begin{array}{c|ccc}
	\delta^{(2)} & y_1^2 & y_1^2y_2 & y_1^3 \\
	\hline
	1  & -1 &  & \\
	x_2  &  &  & 1 \\
	x_2^2  &  & -1 &
\end{array}$$
Matrix $B^{(0)}$ is now invertible; the reduction process is completed. The dimension of $A$ is~$3$. We observe that $\bold{x} = (1, x_2, x_2^{2})$ and $\bold{y} = (y_1, y_1^{2}, y_1^{3})$ are bases of $A$; the associated Horner bases are $\hat{\bold{x}} = (-x_2-x_2^{2}, -x_{2}, 1)$ and $\hat{\bold{y}} = (y_1^{3}, -y_1^{2}-y_1^{2}y_2, -y_1^{2})$.
More generally we have (\cite{jpc} p.57, \cite{bm}, \cite{tm})
\begin{prop}
\label{conjecture}
After the reduction process described above is completed, that is to say when $B^{(0)}$ is invertible and all the matrices $B^{(k)}, k=0, \cdots, n$ have the same size, then the indexing families $\bold{x}$ and $\bold{y}$ are bases of $A$.
\end{prop}

\begin{rem}
Proposition \ref{conjecture} is guaranted when the ideal is zero-dimensional; in this case, to complete the reduction process we just have to use zero-columns of $B^{(0)}$ or, more generally, linear combinations of columns that vanish, i.e elements of the right kernel of $B^{(0)}$.
\end{rem}

\subsubsection{Bezout matrices are related to the companion matrices}
As in Proposition \ref{Barnett}, we consider the matrices $X_1, X_2$ defined by
\begin{equation}
	X_1 = B^{(1)}{B^{(0)}}^{-1} =
	\begin{bmatrix}
		0 & -1 & 0\\
		-1 & 0 & -1\\
		-1 & 0 & 0
	\end{bmatrix},\quad
	X_2 = B^{(2)}{B^{(0)}}^{-1} =
	\begin{bmatrix}
		0 & 0 & 1\\
		1 & 0 & 0\\
		0 & 1 & -1
	\end{bmatrix}
\end{equation}
We can see that $X_1, X_2$ are the multiplication matrices by the variables $x_1, x_2$ in the basis $\bold{x}$, i.e the companion matrices associated to the basis $\bold{x}$. More generally, we have
\begin{prop}
\label{Barnett_multi}
After the reduction process has been completed, the companion matrices $X_j$ and the Bezout matrices are related by the {\bf Barnett formulas}
\begin{equation}
	X_j = B^{(j)}{B^{(0)}}^{-1}, j = 1,\cdots, n
\end{equation}
\end{prop}

\begin{rem}
As in the univariate case, we have, for all $j=1,\cdots,n$,\\
$B^{(j)}{B^{(0)}}^{-1}$ is the multiplication matrix by $x_j$ in the basis $\bold{x}$ \\
${B^{(j)}}^{T}{B^{(0)}}^{-T}$ is the multiplication matrix by $y_j$ in the basis $\bold{y}$ \\
${B^{(0)}}^{-1}{B^{(j)}}$  is the multiplication matrix by $x_j$ in the basis $\hat{\bold{x}}$ \\
${B^{(0)}}^{-T}{B^{(j)}}^{T}$  is the multiplication matrix by $y_j$ in the basis $\hat{\bold{y}}$
\end{rem}

\subsubsection{Numerical computation of the roots}
As in the univariate case, (see Proposition \ref{compan2roots}), the roots of the polynomial system$f_1, \cdots, f_n$ are the eigenvalues of the companion matrices (\cite{AS}). The eigenvalues of matrices $X_1, X_2$ above are
$$
\begin{array}{c|c}
	x_1 & x_2 \\
	\hline
	-1.32472  & 0.75488 \\
	0.66236 + 0.56228i & -0.87744 + 0.74486i \\
	0.66236 - 0.56228i & -0.87744 - 0.74486i
\end{array}
$$
Since $A$ is a commutative algebra, the matrices $X_1, X_2$ commute and have the same eigenvectors. We must be careful to sort the eigenvalues of $X_1, X_2$ so that they correspond to the same eigenvectors. In this example, it is easy to check that the couples $(x_1, x_2)$ are numerical approximations of the roots of the polynomial system $f_1 = x_1^2 + x_1x_2^2 - 1, f_2 = x_1^2x_2 + x_1$ defined in Example \ref{bez_multi}.

\subsection{Numerical experiment}
We provide a reproducible computational experiment that illustrates the effectiveness of the method presented in this article. To perform the computations, follow the steps mentioned in the README.md file of the git repository \cite{jp_code}.
In this experiment, we solve the polynomial system $f = [\\ 
6x_0^2x_1^2x_2x_3^2 + 4x_0^2x_1x_2^2x_3^2 + 9x_0x_1^2x_2^2x_3^2 - 4x_0^2x_2^2x_3^2 - 3x_0x_1x_2^2x_3^2 - x_0^2x_1x_2^2 + 5x_1^2x_2^2x_3 + 6x_0^2x_2x_3^2 - 2x_1x_2^2x_3^2 + x_0^2x_1x_2 - 9x_0x_1^2x_2 + 6x_0x_1x_2x_3 - 3x_2^2x_3^2 - 3x_0x_1^2 - x_0x_1x_2 - 6x_0x_1x_3 - 8x_0x_3^2 + 11x_1x_3^2 + 5x_2x_3^2 + 3x_0^2 - 10x_0x_1 + 8x_1x_2 + 8,\\ -10x_0^2x_1^2x_2^2x_3 + 3x_0^2x_1^2x_2^2 - x_0^2x_1x_2^2x_3 - 12x_0x_1^2x_2^2x_3 + 8x_0x_1x_2^2x_3^2 - x_0^2x_1^2x_2 + 5x_0^2x_1x_2x_3 + x_0^2x_2^2x_3 + 6x_0x_1x_2^2x_3 - 4x_0^2x_2x_3^2 + 6x_1x_2^2x_3^2 - 9x_0x_1^2x_2 + 2x_1^2x_2^2 - x_0^2x_2x_3 + x_1^2x_2x_3 + 4x_1^2x_2 - 2x_1x_2^2 - 8x_0^2x_3 + x_2^2x_3 - 7x_2x_3^2 + 3x_0x_1 - 5x_0x_3 + 10x_1,\\ 7x_0^2x_1^2x_2^2x_3^2 - 6x_0^2x_1x_2^2x_3^2 + 4x_0^2x_1^2x_2x_3 - 11x_0^2x_2^2x_3^2 - 6x_0^2x_1^2x_2 + 3x_0^2x_1x_3^2 - 4x_0x_1^2x_3^2 - 3x_0^2x_2x_3^2 + x_0x_1x_2x_3^2 - 8x_1^2x_2x_3^2 - x_0^2x_2x_3 - 11x_1x_2^2x_3 - 6x_0x_2x_3^2 + 2x_0x_1x_2 + 5x_0^2x_3 - 5x_1x_3^2 - 12x_0x_1 + 2x_1^2 - 6x_0x_2 + 5x_1x_2 + 11x_0x_3 + 8x_2x_3 + 7x_3^2 + 8x_2,\\ 8x_0^2x_1^2x_2^2x_3 - 4x_0x_1^2x_2x_3^2 + 9x_1^2x_2^2x_3^2 + 5x_0^2x_1^2x_2 - 7x_0x_1^2x_2^2 + 4x_0x_1x_2^2x_3 - 5x_1^2x_2^2x_3 + 3x_0^2x_1x_3^2 - x_0x_2^2x_3^2 - 8x_0x_1^2x_2 + x_0x_1x_2^2 + 7x_1^2x_2^2 + 2x_0x_2^2x_3 + 4x_1x_2^2x_3 + 6x_1^2x_2 - 9x_1x_2^2 - 3x_0x_2x_3 - 5x_0x_3^2 + 10x_2x_3^2 - 4x_0x_2 + x_0x_3 + 8x_1x_3 - 7x_2x_3 + 5x_1 + 10$]\\
The size of the Bezout matrix $B^{(0)}$ is $384$; this is the maximum number of solutions that a system of degree [2, 2, 2, 2] can have. To initiate the reduction process, the rank of $B^{(0)}$ is needed; as the matrix has integer coefficients, we use the Sage function matrix.kernel() to calculate its rank. After the reduction process has been completed, we find that the dimension of the quotient $A$ is $352$. The Bezout matrices have rational entries as well as the companion matrices $X_j = B^{(j)}{B^{(0)}}^{-1}$ ; computing the numerical eigenvalues of the $X_j$'s give approximations of the roots of the polynomial system~$f$.

\subsubsection{Checking the result}
\label{checking_the_result}
We check the result of the calculations by two means. 
\begin{enumerate}
\item
We compare the dimension of the quotient $A$, as provided by the Groebner basis, to the dimension provided by the Bezout process, i.e the size of the reduced Bezout matrices. If, as it is the case in this experiment, the polynomial system is in complete intersection, then the two dimensions coincide. If, on the contrary, the polynomial system is not in complete intersection, then the dimension provided by the Groebner basis is infinite whereas the dimension provided by the Bezout process is always finite.
\item
We check that $f(X) = 0$, i.e $f_i(X_1,\cdots, X_n) = 0$ for every $i = 1\cdots, n$ and where the $X_j$'s are the companion matrices. When the polynomial system is in complete intersection, this means that every $f_i$ belongs to the ideal $\langle f\rangle$. But, in practice, since the companion matrices have rational entries, the calculation of each $f_i(X_1,\cdots, X_n)$ may be too heavy. Rather, we choose to make the computations in a finite field $F$, typically $\mathbb{Z}/ p\mathbb{Z}$ with, in this experiment, $p = 2003$. Moreover, to lighten even more the computation, we choose a random vector $v\in F$ and compute only the product of each $f_i(X_1,\cdots, X_n)$ by $v$. Thus, we only need to do matrix-vector products. If at the end of the process, and for all $i= 1\cdots, n$, we have $f_i(X_1,\cdots, X_n)v = 0$ then we can conclude that, with a very high probability, $f(X) = 0$. In this experiment we know that the test was succesful because in the sage variable test\_XX, a Python list, all the entries equal True.
\item
We check the quality of the numerical roots $\alpha$ by computing the errors $Df(\alpha)^{-1}f(\alpha)$. An histogram of these errors is shown in Table \ref{tab:histo} (the errors are calculated only when the jacobian $Df(\alpha)$ is invertible).
\end{enumerate}

\begin{table}[h]
\begin{center}
\begin{tabular}{c|c}
 log10 of errors & nb of roots \\ 
 \hline
 $[-14.5, -13.9]$ & $29$\\
$[-13.9, -13.2]$ & $95$\\
$[-13.2, -12.6]$ & $82$\\
$[-12.6, -12.0]$ & $68$\\
$[-12.0, -11.3]$ & $31$\\
$[-11.3, -10.7]$ & $9$\\
$[-10.7, -10.0]$ & $18$\\
$[-10.0, -9.4]$ & $9$\\
$[-9.4, -8.8]$ & $4$\\
$[-8.8, -8.1]$ & $1$\\
$[-8.1, -7.5]$ & $4$\\
$[-6.8, -6.2]$ & $1$\\
$[-3.7, -3.0]$ & $1$

\end{tabular}
\end{center}
\caption{histogram of errors}
\label{tab:histo}
\end{table}

\subsubsection{Timings}
Table \ref{tab:timings} shows the timings of the Bezout computations as compared to the timings of the Groebner computations.
\begin{table}[h]
\begin{center}
\begin{tabular}{llllr}
 Method & Computation & Software & Arithmetic & Timing \\ \hline
   \multirow{4}{*}{Bezout} & Bezout matrices & NumPy & floating point & $2.6907$ s \\
   & rank of $B(1)$ via rref()  & Sage & integer & $0.1208$ s \\
   & reduction process & Sage & integer & $7.8421$ s \\
   & eigenvalues & SciPy & floating point & $3.4766$ s \\ 
   \hline
   \hline
   Groebner & Groebner basis & Sage & integer & $1923.9460$ s
\end{tabular}
\end{center}
\caption{timings}
\label{tab:timings}
\end{table}

\subsubsection{Disk space usage}
Table \ref{tab:sizes} shows the disk space used by the Bezout matrices (after the reduction process; only the non-zero entries are taken in account) as compared to the disk space used by the Groebner basis.
\begin{table}[h]
\begin{center}
\begin{tabular}{llc}
 Method & Output & Disk space usage \\ \hline
   \multirow{2}{*}{Bezout} & reduced Bezout matrices & $1.5240$ Mo\\
   & (non-zero entries only) & \\
   \hline
   \hline
   Groebner & Groebner basis & $156.4239$ Mo
\end{tabular}
\end{center}
\caption{disk space usage}
\label{tab:sizes}
\end{table}

\subsection{What happens when the polynomial system is not zero-dimensional ?}
It may happen that the square polynomial system is not zero-dimensional ; the dimension of $A$ is then infinite and this is indicated by the Groebner calculation. The Bezout process, no matter the dimension of $A$ is finite or infinite, always produces matrices $B^{(j)}$ of finite size ; in the case when the dimension is infinite, we still call the matrices $X_j = B^{(j)}{B^{(0)}}^{-1}$ companion matrices.
What we observe in our experiments is that, whenever the dimension of $A$ is infinite,  we still have $f(X) = 0$ (item 2 of \ref{checking_the_result}) and the eigenvalues of the companion matrices are still roots of $f$ (item 3 of \ref{checking_the_result}).

\end{document}